\documentclass[11pt,leqno]{article}
\usepackage{amsmath} 
\usepackage[mathscr]{eucal}
\usepackage{amssymb} 
\usepackage{theorem}
\usepackage{xcolor}

\usepackage{enumerate}

\theoremstyle{change}  
\newtheorem{theorem}{Theorem}[section] 

\newtheorem{corollary}[theorem]{Corollary}
\theorembodyfont{\rmfamily}  
\newtheorem{conjecture}[theorem]{Conjecture}

\newtheorem{remark}[theorem]{Remark}

\newtheorem{notation}[theorem]{Notation}
\newtheorem{nothing}[theorem]{} 

\newenvironment{proof}{\noindent{\bf Proof}\ }{\qed\bigskip}

\renewcommand{\le}{\leqslant} 

\def\Aut{\mathrm{Aut}}

\def\BL{\mathcal{B}\ell_k}
\def\FBL{\calF\mathcal{B}\ell}

\newcommand{\calC}{\mathcal{C}}

\newcommand{\calF}{\mathcal{F}}

\newcommand{\calS}{\mathcal{S}}

\newcommand{\Ind}{\mathrm{Ind}}

\newcommand{\Inn}{\mathrm{Inn}}

\newcommand{\lexp}[2]{\setbox0=\hbox{$#2$} \setbox1=\vbox to
                 \ht0{}\,\box1^{#1}\!#2}
\newcommand{\Out}{\mathrm{Out}}
\newcommand{\qed}{\nobreak\hfill
                  \vbox{\hrule\hbox{\vrule\hbox to 5pt
                  {\vbox to 8pt{\vfil}\hfil}\vrule}\hrule}}

\newcommand{\Z}{\mathbb{Z}}
\newcommand{\RMod}[1]{#1\hbox{-}\mathrm{Mod}}

\newcommand{\Br}{\mathrm{Br}}

\def\sur{\overline}

\def\F{\mathbb{F}}  
\def\k{k}           
\def\un{\mathbf{1}} 
\def\dom{\backslash}
\def\SR{\mathsf{R}}
\def\CL{\mathcal{L}}
\def\sur{\overline}

\def\Ind{\mathrm{Ind}}
\def\Br{\mathrm{Br}}
\def\br{\mathrm{Br}}

\newcommand{\gen}[1]{\langle #1\rangle}
\newcommand{\semid}[2]{#1{\rtimes}\gen{#2}}
\newcounter{parag}[section]
\newcommand{\tpar}{\medskip\par\noindent\pagebreak[3]\refstepcounter{theorem}\refstepcounter{parag}{\bf \thesection.\theparag.\ }}
\def\ldbrack{[\![}
\def\rdbrack{]\!]}

\title{On Alperin's conjecture and functorial equivalence of blocks}
\author{Robert Boltje, Serge Bouc, and Deniz Y\i lmaz}

\begin{document}
\sloppy
\maketitle

\begin{abstract} Let $k$ be an algebraically closed field of positive characteristic $p$ and let $\F$ be an algebraically closed field of characteristic 0. We consider Alperin's weight conjecture (over $k$) from the point of view of (stable) functorial equivalence of blocks over $\F$. We formulate a functorial version of Alperin's blockwise weight conjecture, and show that it is equivalent to the original one. We also show that this conjecture holds {\em stably}, i.e. in the category of stable diagonal $p$-permutation functors over~$\F$. 
\end{abstract}
{\flushleft{\bf MSC2020:}} 20C20, 20J15, 19A22. 
{\flushleft{\bf Keywords:}} blocks of group algebras, Alperin's weight conjecture, diagonal $p$-permutation functor, functorial equivalence.

\section{Introduction}
Let $k$ be an algebraically closed field of positive characteristic $p$. The first step in studying the representation theory of a finite group $G$ over $k$ consists in splitting the group algebra $kG$ as a direct product of the {\em block algebras} $kGb$, where $b$ is a {\em block idempotent} of $kG$, that is, a primitive idempotent of the center of $kG$. In this paper, such a pair $(G,b)$ of a finite group $G$ and a block idempotent $b$ of $kG$ is called a {\em group-block} pair (over $k$).\par
To speak very loosely, block theory is the study of such group-block pairs $(G,b)$, and of the numerous invariants and structures attached to them: The block algebra $kGb$ itself, in the first place, and its various - derived or stable - categories of modules. But also invariants of a more combinatorial nature, such as defect groups, source algebras or fusion systems. For each of these invariants, there is a corresponding notion of equivalence of group-block pairs: Morita equivalence, splendid Rickard equivalence, Puig equivalence, stable equivalence of Morita type, $p$-permutation equivalence, isotypy, perfect isometry (\cite{Rickard96}, \cite{Linckelmann}, \cite{BoltjePerepelitsky2025}, \cite{Broue90})\ldots~All these equivalences have been the object of intense research, and a huge amount of literature is now devoted to them.\par
These invariants also led to important conjectures, of various kinds. Some of them, called {\em finiteness conjectures}, predict that there are only {\em finitely many} group-block pairs up to some particular equivalence, once a specific invariant is given: For example, Donovan's conjecture (\cite[Conjecture~6.1.9]{Linckelmann}) states that there are only finitely many block algebras of group-block pairs with a given defect group, up to Morita equivalence. Other types of conjectures - called {\em local-global} or {\em counting conjectures} - assert that some invariants for a group-block pair $(G,b)$ can be computed from the invariants attached to ``smaller'' pairs and ``{\em local data}'' - generally obtained by applying the Brauer morphism. This is in particular the case of Alperin's weight conjecture, in its blockwise form (\cite{Alperin1987}). Finally, some {\em structural  conjectures}, like Brou\'e's abelian defect group conjecture, predict the existence of some specific equivalence between derived categories of modules associated to group-block pairs related by the Brauer correspondence (\cite{Broue90}). All these conjectures have been verified in numerous cases, and proved for some particular classes of group-block pairs, but they remain essentially open.\par
Further new invariants were attached recently (\cite{BoucYilmaz2022}, \cite{BoucYilmaz2024}) to group-block pairs, in the form of  {\em (stable) diagonal $p$-permutation functors} over some commutative ring $\SR$ (see Section~\ref{review} for details), leading to a corresponding notion of {\em (stable) functorial equivalence} of group-block pairs over $\SR$. These invariants lead in particular to a finiteness theorem (\cite{BoucYilmaz2022}, Theorem 10.6), in the spirit of Donovan's conjecture: If $\F$ is an algebraically closed field of characteristic~0, and $D$ is a finite $p$-group,  then there is only a finite number of group-block pairs $(G,b)$ with defect groups isomorphic to $D$, up to functorial equivalence over $\F$.\par
In the present paper, we consider Alperin's weight conjecture from the point of view of (stable) diagonal $p$-permutation functors. In Section 3, we first formulate (Conjecture~\ref{conjecture}) a functorial form FAwc$(G,b)$ of Alperin's conjecture Awc$(G,b)$, for each group-block pair $(G,b)$. The main theorem of this section is Theorem~\ref{thm Alperin reformulation}, stating that Conjecture FAwc$(G,b)$ holds {\em stably}, i.e. in the Grothendieck group of the category of stable diagonal $p$-permutation functors over $\F$. As a consequence, our conjecture FAwc$(G,b)$ is in fact {\em equivalent} to Awc$(G,b)$.

It has been suspected by many experts that the original formulation of Alperin's weight conjecture, namely the equality of two numbers associated to $(G,b)$, is only the shadow of a more structural (yet hidden) phenomenon. We hope that FAwc$(G,b)$ is a first step towards a more structural explanation.

\begin{notation}
Let $G$ be a group. By $i_g\colon G\to G$, we denote the conjugation map $x\mapsto gxg^{-1}$. Moreover, for $x\in G$ and $H\le G$ we set $\lexp{g}{x}:=i_g(x)$ and $\lexp{g}{H}:=i_g(H)$.
If $X$ is a left $G$-set, the stabilizer of an element $x\in X$ is denoted by $G_x$. By $G\backslash X$ we denote the set of $G$-orbits of $X$, and by $[G\backslash X]$ we denote a set of representatives of the $G$-orbits of $X$.

\end{notation}


\section{Review of (stable) functorial equivalence of blocks}\label{review}
\tpar Let $\SR$ be a commutative ring (with 1), and $\k$ be an algebraically closed field of characteristic $p>0$. We denote by $\SR pp_k^\Delta$ the category introduced in~\cite{BoucYilmaz2020}, where objects are finite groups, and the set of morphisms from a group $G$ to a group $H$ is $\SR T^\Delta(H,G)=\SR\otimes_\Z T^\Delta(H,G)$, where $T^\Delta(H,G)$ is the Grothendieck group of {\em diagonal $p$-permutation $(kH,kG)$-bimodules}. \par
A {\em diagonal $p$-permutation functor} is by definition (see~\cite{BoucYilmaz2020}) an $\SR$-linear functor from  $\SR pp_k^\Delta$ to the category $\RMod{\SR}$ of all $\SR$-modules. These functors, together with their natural transformations,  form an abelian category which we simply denote by $\calF_{\SR,k}$ (instead of $\calF_{\SR pp_k}^\Delta$ as in \cite{BoucYilmaz2020}). 
\tpar We denote by $\sur{\SR pp_k^\Delta}$ the quotient category of $\SR pp_k^\Delta$ by the morphisms that factor through the trivial group. A {\em stable diagonal $p$-permutation functor} (see~\cite{BoucYilmaz2024}) is an $\SR$-linear functor from  $\sur{\SR pp_k^\Delta}$ to $\RMod{\SR}$, or equivalently, a diagonal $p$-permutation functor which vanishes at the trivial group. Stable diagonal $p$-permutation functors also form an abelian category, that we simply denote by $\sur{\calF_{\SR,k}}$. If $F$ is a diagonal $p$-permutation functor, we denote by $\sur{F}$ its largest stable quotient, i.e. the quotient of $F$ by the subfunctor generated by $F(\un)$. In other words $\sur{F}(G)=F(G)/\SR T^\Delta(G,\un)F(\un)$ for any finite group $G$ (\cite{BoucYilmaz2024}, Remark 3.4).
\tpar Let $(G,b)$ be a pair of a finite group $G$ and a central idempotent $b$ of $kG$ (recall that when $b$ is moreover {\em primitive}, the pair $(G,b)$ is called a {\em group-block} pair). Then the (isomorphism class of the) $(kG,kG)$-bimodule $kGb$ is an idempotent endomorphism of $G$ in $\SR pp_k^\Delta$. The diagonal $p$-permutation functor $\SR T_{G,b}^\Delta$ {\em associated to $(G,b)$} is the corresponding direct summand $\SR T^\Delta(-,G)\circ kGb$ of the representable functor at $G$ obtained by precomposition with $kGb$.\par
We say that two group-block pairs $(G,b)$ and $(H,c)$ are {\em functorially equivalent over~$\SR$} (\cite{BoucYilmaz2022},~Definition 10.1) if the functors $\SR T^\Delta_{G,b}$ and $\SR T^\Delta_{H,c}$ are isomorphic in $\calF_{\SR,k}$. Similarly, we say that $(G,b)$ and $(H,c)$ are {\em stably functorially equivalent over $\SR$} if the functors $\sur{\SR T^\Delta_{G,b}}$ and $\sur{\SR T^\Delta_{H,c}}$ are isomorphic in $\sur{\calF_{\SR,k}}$.
\tpar We denote by $\SR \BL$ the (partial\footnote{In the classical definition of the idempotent completion of a category $\calC$, the objects are {\em all} the pairs $(X,e)$ of an object $X$ of $\calC$ and an idempotent endomorphism $e$ of $X$ in $\calC$. Here we consider only {\em some} of these pairs. This does not affect the main properties of the idempotent completion.}) idempotent completion of $\SR pp_k^\Delta$ constructed from blocks of group algebras, i.e. the category defined as follows:
\begin{itemize}
\item The objects of $\SR \BL$ are pairs $(G,b)$, where $G$ is a finite group and $b$ is a central idempotent of $kG$. 
\item The set of morphisms from $(G,b)$ to $(H,c)$ in $\SR \BL$ is the subset of $\SR T^\Delta(H,G)$ obtained by precomposition with $kGb$ and postcomposition with $kHc$, in other words the set {$kHc\,\circ\, \SR T^\Delta(H,G)\,\circ\, kGb$}.
\item The composition of morphisms $u:(G,b)\to (H,c)$ and $v:(H,c)\to (K,d)$ in $\SR \BL$ is induced by the tensor product of bimodules over $kH$.
\item The identity morphism of $(G,b)$ in $\SR \BL$ is (the isomorphism class of) the $(kG,kG)$-bimodule $kGb$.
\end{itemize}
Note that the objects $(G,0)$ in $\SR \BL$ are $0$-objects.
Similarly, we denote by $\sur{\SR\BL}$ the quotient category of $\SR\BL$ by the morphisms which factor through the trivial pair $(\un,1_k)$. Equivalently $\sur{\SR\BL}$ is the partial idempotent completion built as above from the category $\sur{\SR pp_k^\Delta}$.\par
We denote by $\FBL_{\SR,k}$ (resp. $\sur{\FBL_{\SR,k}}$) the category of $\SR$-linear functors from $\SR\BL$ (resp. $\sur{\SR\BL}$) to $\RMod{\SR}$.
By standard results, the inclusion functor $\SR pp_k^\Delta\to\SR\BL$ (resp. $\sur{\SR pp_k^\Delta}\to\sur{\SR\BL}$) sending a group $G$ to the pair $(G,1_{kG})$ induces by composition an equivalence of categories from $\FBL_{\SR,k}$ to $\calF_{\SR,k}$ (resp. from $\sur{\FBL_{\SR,k}}$ to $\sur{\calF_{\SR,k}}$). We observe that if a central idempotent $b$ of $kG$ is an orthogonal sum of two central idempotents $b_1$ and $b_2$, the functor $\SR T^\Delta_{G,b}$ is naturally isomorphic to the direct sum $\SR T^\Delta_{G,b_1}\oplus \SR T^\Delta_{G,b_2}$ in $\calF_{R,k}$. We also observe that two group-block pairs $(G,b)$ and $(H,c)$ are functorially equivalent (resp. stably functorially equivalent) over $\SR$ if and only if $(G,b)$ and $(H,c)$ are isomorphic in $\SR\BL$ (resp. in $\sur{\SR\BL}$). 

\begin{nothing}\label{semisimplicity} 
Let $\F$ be an algebraically closed field of characteristic 0. It was shown in~\cite{BoucYilmaz2022} and~\cite{BoucYilmaz2024} that the categories $\calF_{\F,k}$ and $\sur{\calF_{\F,k}}$ are semisimple, and it follows that the categories $\FBL_{\F,k}$ and $\sur{\FBL_{\F,k}}$ are also semisimple. We denote by $K_0(\FBL_{\F,k})$ and $K_0(\sur{\FBL_{\F,k}})$ their respective Grothendieck groups.\par
The simple diagonal $p$-permutation functors over $\F$ are parametrized by means of {\em $D^\Delta$-pairs}: By definition (\cite{BoucYilmaz2022}, Definition~3.2), this is a pair $(L,u)$ of a finite $p$-group $L$ and a generator $u$ of a $p'$-group acting faithfully on $L$ - in other words $u$ is not only a $p'$-element acting on $L$, but a $p'$-{\em automorphism} of $L$. An isomorphism $\varphi:(L,u)\to(M,v)$ of $D^\Delta$-pairs is a group isomorphism $\varphi:\semid{L}{u}\to\semid{M}{v}$ between the corresponding semidirect products, which sends $u$ to a conjugate of $v$. We denote by $\Aut(L,u)$ the group of automorphisms of the pair $(L,u)$, and by $\Out(L,u)$ the quotient of $\Aut(L,u)$ by the subgroup $\Inn(\semid{L}{u})$ of inner automorphisms of $\semid{L}{u}$. \par 
The simple diagonal $p$-permutation functors $S_{L,u,V}$ over $\F$ (up to isomorphism) are then parametrized by triples $(L,u,V)$ (up to isomorphism), where $(L,u)$ is a $D^\Delta$-pair, and $V$ is a simple $\F \Out(L,u)$-module. The simple stable diagonal $p$-permutation functors are the functors $S_{L,u,V}$, where $L$ is a {\em nontrivial} $p$-group. So the group $K_0(\FBL_{\F,k})$ has a $\Z$-basis consisting of the isomorphism classes $[S_{L,u,V}]$ of simple functors $S_{L,u,V}$, and 
$K_0(\sur{\FBL_{\F,k}})$ has a $\Z$-basis consisting of the classes $[S_{L,u,V}]$ with $L\neq \un$.\par
When $G$ is a finite group, and $b$ is a central idempotent of $kG$, we simply denote by $\ldbrack G,b\rdbrack_\F$ the image of the functor $\F T^\Delta_{G,b}$ in $K_0(\FBL_{\F,k})$, and by $\sur{\ldbrack G,b\rdbrack}_\F$ its image in $K_0(\sur{\FBL_{\F,k}})$.
\end{nothing}


\section{Alperin's weight conjecture and stable functorial equivalence}\label{alperin}
By Kn\"orr-Robinson, see \cite[Theorem~3.8]{KnorrRobinson}, Alperin's blockwise weight conjecture, that we refer to as AWC, is equivalent to saying that the following conjecture Awc$(G,b)$ holds for all group-block pairs $(G,b)$.

\begin{conjecture}\label{AWC}
Let $(G,b)$ be a group-block pair over $k$. Then
\begin{equation*}
  \sum_{\sigma\in [G\backslash \calS_p(G)]}  (-1)^{|\sigma|}  \, l(\k G_\sigma b_\sigma) = 
   \begin{cases} 
           1 & \text{if $d(b)=0$}; \\ 0  & \text{if $d(b)>0$.}
   \end{cases}\tag{Awc$(G,b)$}
\end{equation*}
\end{conjecture}

Here: \begin{itemize}
\item $\calS_p(G)$ denotes the set of strictly ascending chains $(\un=P_0<P_1<\cdots<P_n)$ of $p$-subgroups of $G$.
\item $G_\sigma$ is the stabilizer in $G$ of $\sigma$. 
\item $|\sigma|=n$ if $\sigma= (\un=P_0<P_1<\cdots<P_n)\in \calS_p(G)$.
\item For $\sigma$ as above, $b_\sigma:=\br_{P_n}(b)$, where $\br_{P_n}$ is the Brauer homomorphism with respect to $P_n$.
\item $l$ associates to a finite-dimensional $\k $-algebra the number of its simple modules (up to isomorphism). 
\item $d(b)$ denotes the defect of $b$. 
\end{itemize}

\begin{remark}\label{stabilizer block}
One can show  that $b_\sigma$ is actually a sum of block idempotents of $\k G_\sigma$ and that it does not depend on the chain $\sigma$ but only on the stabilizer $G_\sigma$ (see Lemma~3.1 and the following Remark in \cite{KnorrRobinson}). If $b$ is a block of $kG$ with trivial defect group, Conjecture Awc$(G,b)$ holds trivially. 
\end{remark}

\bigskip


For a group-block pair $(G,b)$, we propose the following conjecture, denoted by FAwc$(G,b)$ (with ``F" standing for functorial), using the notation of Conjecture~\ref{AWC}:

\begin{conjecture}\label{conjecture}
Let $(G,b)$ be a group-block pair over $k$. Then, in $K_0(\FBL_{\F,k})$, we have
\begin{equation*}
  \sum_{\sigma\in [G\backslash \calS_p(G)]}  (-1)^{|\sigma|}  \, \ldbrack G_\sigma, b_\sigma\rdbrack_\F = 
   \begin{cases} 
           [S_{\un,1,\F}] & \text{if $d(b)=0$}; \\ 0  & \text{if $d(b)>0$.}
   \end{cases}\tag{FAwc$(G,b)$}
\end{equation*}
\end{conjecture}

Again, if $b$ has trivial defect group, it is easy to show that Conjecture~FAwc$(G,b)$ holds, see \cite[Corollary~8.23]{BoucYilmaz2022}. The statement that FAwc$(G,b)$ holds for all group-block pairs $(G,b)$ is abbreviated by FAWC.



\medskip
It is straightforward to see that for a group-block pair $(G,b)$, Conjecture FAwc$(G,b)$ implies Conjecture Awc$(G,b)$:

\begin{theorem} \label{FAWC implies AWC}\begin{enumerate}
\item Let $(G,b)$ be a group-block pair over $k$. If Conjecture {\rm FAwc}$(G,b)$ holds, then Conjecture {\rm Awc}$(G,b)$ holds. 
\item In particular, Conjecture {\rm FAWC} implies Conjecture~{\rm AWC}.
\end{enumerate}
\end{theorem}

\begin{proof} Since $\F T^\Delta_{G,b}(\un)$ is isomorphic to the $\F$-vector space spanned by the indecomposable projective $kGb$-modules, one has $l(\k Gb)=\dim_\F \F T^\Delta_{G,b}(\un)$, and by \cite[Corollary 8.23]{BoucYilmaz2022} the latter is equal to the multiplicity of the simple functor $S_{\un,1,\F }$ in $\F T^\Delta_{G,b}$. Assertion  1 is now immediate by considering the multiplicity of $[S_{\un,1,\F}]$ in both sides of FAwc$(G,b)$, and then Assertion 2 follows.
\end{proof}

In the following theorem, the notation is the same as in Conjecture~\ref{AWC}:
\begin{theorem}\label{thm Alperin reformulation}\begin{enumerate}
\item Let $(G,b)$ be a group-block pair over $k$. Then there exists an integer $n_{G,b}$ such that, in $K_0(\FBL_{\F,k})$,
$$\sum_{\sigma\in [G\backslash \calS_p(G)]}(-1)^{|\sigma|}\ldbrack G_\sigma,b_\sigma\rdbrack_\F=n_{G,b}[S_{\un,1,\F}].$$
\item In particular, Conjecture {\rm FAWC} is equivalent to {\rm AWC}, and for any group-block pair $(G,b)$ over $k$, Conjecture {\rm FAwc}(G,b) is equivalent to {\rm Awc}(G,b).
\end{enumerate}
\end{theorem}

\begin{proof}
1. Let $\Sigma_{G,b}\in K_0(\FBL_{\F,k})$ denote the alternating sum in the left hand side of FAwc$(G,b)$, that is,
$$\Sigma_{G,b}:=\sum_{\sigma\in [G\backslash \calS_p(G)]}(-1)^{|\sigma|}\ldbrack G_\sigma,b_{\sigma}\rdbrack_\F\,.$$
We want to show that the multiplicity of a simple functor $S_{L,u,V}$ in $\Sigma_{G,b}$ is equal to 0 if $L\neq 1$. By \cite[Theorem~8.22]{BoucYilmaz2022}, we know that the multiplicity $m_{L,u,V}(G_\sigma,b_\sigma)$ of $S_{L,u,V}$ as a composition factor of the functor 
$\F T^\Delta_{G,b}$ is given by
$$m_{L,u,V}(G_\sigma,b_\sigma)=\sum_{(P_\gamma,\pi)\in[G_\sigma\dom \CL_{b_\sigma}(G_\sigma,L,u)/\Aut(L,u)]}\dim_\F V^{\Aut(L,u)_{\sur{(P_\gamma,\pi)}}},$$
where the notation is as follows:
\begin{itemize}
\item $\CL_{b_\sigma}(G_\sigma,L,u)$ is the set of pairs $(P_\gamma,\pi)$ of a local point $P_\gamma$ on $kG_\sigma b_\sigma$, i.e., a $p$-subgroup $P$ of $G_\sigma$ and a conjugacy class $\gamma$ of primitive idempotents of $(kG_\sigma b_\sigma)^P$, and $\pi:L\to P$ is a group isomorphism such that $\pi u \pi^{-1}\in N_{G_\sigma}(P_\gamma)$, i.e., such that there exists $g\in N_G(P_\gamma)$ with $i_g\pi=\pi u$.
\item The set $\CL_{b_\sigma}(G_\sigma,L,u)$ is a $\big(G_\sigma,\Aut(L,u)\big)$-biset via the action defined by
$$g\cdot(P_\gamma,\pi)\cdot\varphi=({^gP}_{^g\gamma},i_g\pi\varphi)\,,$$
for $(g,\varphi)\in G_\sigma\times \Aut(L,u)$ and $(P_\gamma,\pi)\in \CL_{b_\sigma}(G_\sigma,L,u)$. For $(P_\gamma,\pi)\in \CL_{b_\sigma}(G_\sigma,L,u)$, we denote by $\sur{(P_\gamma,\pi)}$ the left orbit $G_\sigma(P_\gamma,\pi)$, and by $\Aut(L,u)_{\sur{(P_\gamma,\pi)}}$ the stabilizer of this orbit in $\Aut(L,u)$, namely
$$\Aut(L,u)_{\sur{(P_\gamma,\pi)}}=\{\varphi\in\Aut(L,u)\mid\exists g\in N_{G_\sigma}(P_\gamma),\, i_g\pi=\pi\varphi\}.$$
\end{itemize}
So we want to show that
$$\sum_{\sigma\in [G\backslash \calS_p(G)]}(-1)^{|\sigma|}m_{L,u,V}(G_\sigma,b_\sigma)=0$$
whenever $L\neq 1$. We observe that
$$\dim_\F V^{\Aut(L,u)_{\sur{(P_\gamma,\pi)}}}=\big(\Ind_{\Aut(L,u)_{\sur{(P_\gamma,\pi)}}}^{\Aut(L,u)}\F,V\big)_{\Aut(L,u)}\,,$$
where $(-,-)_{\Aut(L,u)}$ denotes the Schur inner product on the Grothendieck group (character ring) $R_\F(\Aut(L,u))$ of finite dimensional $\F\Aut(L,u)$-modules.
We set 
$$W(G,L,u)=\sum_{\sigma\in [G\backslash \calS_p(G)]}(-1)^{|\sigma|}\sum_{(P_\gamma,\pi)\in[G_\sigma\dom \CL_{b_\sigma}(G_\sigma,L,u)/\Aut(L,u)]}\Ind_{\Aut(L,u)_{\sur{(P_\gamma,\pi)}}}^{\Aut(L,u)}\F,$$
which we view as an element in $R_\F(\Aut(L,u))$. We want to show that 
$$\big(W(G,L,u),V\big)_{\Aut(L,u)}=0\,,$$for all $L\neq 1$ and all simple $\F \Aut(L,u)$-modules $V$. But this amounts to saying that the virtual character $W(G,L,u)$ is equal to 0. Since $W(G,L,u)$ is a (virtual) permutation character its value at $\varphi\in\Aut(L,u)$ is equal to
$$|W(G,L,u)^\varphi|=\sum_{\sigma\in [G\backslash \calS_p(G)]}(-1)^{|\sigma|}\big|\big(G_\sigma\dom \CL_{b_\sigma}(G_\sigma,L,u)\big)^\varphi\big|\,.$$
So, we want to show that this number is equal to 0 if $L\neq 1$.\par
Summing over all $\sigma\in \calS_p(G)$ rather than representatives of $G$-orbits, we have that
\begin{align*}
|W(G,L,u)^\varphi|&=\sum_{\sigma\in \calS_p(G)}\frac{|G_\sigma|}{|G|}(-1)^{|\sigma|}\big|\big(G_\sigma\dom \CL_{b_\sigma}(G_\sigma,L,u)\big)^\varphi\big|\\
&=\sum_{\substack{\sigma\in \calS_p(G)\\(P_\gamma,\pi)\in  \CL_{b_\sigma}(G_\sigma,L,u)\\\,G_\sigma(P_\gamma,\pi)\cdot\varphi=G_\sigma(P_\gamma,\pi)}}\frac{|G_\sigma|}{|G|}(-1)^{|\sigma|}\frac{|G_\sigma\cap G_{(P_\gamma,\pi)}|}{|G_\sigma|}\\
&=\sum_{\substack{\sigma\in \calS_p(G)\\(P_\gamma,\pi)\in  \CL_{b_\sigma}(G_\sigma,L,u)\\\,G_\sigma(P_\gamma,\pi)\cdot\varphi=G_\sigma(P_\gamma,\pi)}}(-1)^{|\sigma|}\frac{|G_\sigma\cap G_{(P_\gamma,\pi)}|}{|G|}
\end{align*}
where $G_{(P_\gamma,\pi)}$ is the left stabilizer of $(P_\gamma,\pi)$, i.e.,
$$G_{(P_\gamma,\pi)}=\{g\in G\mid {^gP}=P,\,{^g\gamma}=\gamma,\,i_g\pi=\pi\}=C_G(P)\cap N_G(P_\gamma).$$
Now $G_\sigma(P_\gamma,\pi)\cdot\varphi=G_\sigma(P_\gamma,\pi)$ if and only if there exists $g\in G_\sigma$ such that $(^gP_{^g\gamma},i_g\pi)=(P_\gamma,\pi\varphi)$, and in this case, the number of such elements $g\in G_\sigma$ is equal to $|G_\sigma\cap G_{(P_\gamma,\pi)}|$. It follows that
$$|W(G,L,u)^\varphi|=\frac{1}{|G|}\sum_{\substack{(\sigma,P_\gamma,\pi,g)\\\sigma\in \calS_p(G)\\P_\gamma\in \CL_{b_\sigma}(G_\sigma,L,u)\\g\in G_\sigma\\(^gP_{^g\gamma},i_g\pi)=(P_\gamma,\pi\varphi)}}(-1)^{|\sigma|}\,.$$
We can rewrite this as
$$|W(G,L,u)^\varphi|=\frac{1}{|G|}\sum_{(\sigma,P,\gamma,\pi,g)\in\mathbb{S}}(-1)^{|\sigma|},$$
where $\mathbb{S}$ is the set of quintuples $(\sigma,P,\gamma,\pi,g)$ such that:
\begin{itemize}
\item $\sigma\in \calS_p(G)$,
\item $P\le G_\sigma$,
\item $\gamma$ is a local point of $(kG_\sigma b_\sigma)^P$,
\item $\pi:L\stackrel{\cong}{\to}P$ is a group isomorphism such that $\pi u \pi^{-1}\in N_{G_\sigma}(P_\gamma)$,
\item $g\in G_\sigma$ is such that $(^gP_{^g\gamma},i_g\pi)=(P_\gamma,\pi\varphi)$, in other words $g\in N_{G_\sigma}(P_\gamma)$ and $i_g\pi=\pi\varphi$.
\end{itemize}
The proof that $|W(G,L,u)^\varphi|=0$ is inspired by the proof of Lemma 4.1 of~\cite{KnorrRobinson}. We will build an involution $(\sigma,P,\gamma,\pi,g)\mapsto (\sigma',P,\gamma',\pi,g)$ of $\mathbb{S}$ such that $|\sigma'|=|\sigma|\pm1$.\par

Let $(\sigma,P,\gamma,\pi,g)\in\mathbb{S}$, with $\sigma=(1=P_0<P_1<\ldots<P_n)$. Since $P\cong L\neq 1=P_0$, there is a largest integer $i\in\{0,\ldots,n\}$ such that $P\not\le P_i$. There are two cases:
\begin{itemize}
\item Either $i=n$ or $PP_i<P_{i+1}$, then set $\sigma'=\sigma\sqcup\{PP_i\}$, i.e.,
$$\sigma'=(P_0<\ldots <P_n< PP_n)\quad \text{or}\quad \sigma'=(P_0<P_1<\ldots P_i<PP_i<P_{i+1}<\ldots<P_n)\,,$$ respectively.
\item Or $PP_i=P_{i+1}$, and then set $\sigma'=\sigma\smallsetminus \{P_{i+1}\}$, i.e.,
$$\sigma'=(P_0<P_1<\ldots P_{i}<P_{i+2}<\ldots<P_n).$$
\end{itemize}
One checks easily that $(\sigma')'=\sigma$, and it is clear that $|\sigma'|=|\sigma|\pm1$. Moreover $P\le G_{\sigma'}$ if and only if $P\le G_\sigma$, and $N_G(P)\cap G_\sigma= N_G(P)\cap G_{\sigma'}$, i.e., $N_{G_\sigma}(P)=N_{G_{\sigma'}}(P)$. \par
Now by \cite[Corollary~37.6]{Thevenaz95}, the Brauer morphism $\Br_P^{G_\sigma}$ induces a bijection between the local points of $(kG_\sigma b_\sigma)^P$ and the points of $kC_{G_\sigma}(P)\Br_P^{G_\sigma}(b_\sigma)$. Since $N_{G_\sigma}(P)=N_{G_{\sigma'}}(P)$, we also have $C_{G_\sigma}(P)=C_{G_{\sigma'}}(P)$. Moreover, $\Br_P^{G_\sigma}(b_\sigma)=\Br_P^{G_{\sigma'}}(b_{\sigma'})$. In fact, $\Br_P^{G_\sigma}(b_\sigma)$ is the truncation to $kC_{G_\sigma}(P)$ of $b_{\sigma}\in kC_{G_\sigma}(P_n)$, so $\Br_P^{G_\sigma}(b_\sigma)$ is the truncation of $b$ to $kC_{G_\sigma}(PP_n)=kC_{G_{\sigma'}}(PP_n)$, which is equal to $\Br_P^{G_{\sigma'}}(b_{\sigma'})$. \par
It follows that the Brauer morphism at $P$ induces a bijection $\gamma\mapsto \gamma'$ between the local points of $(kG_\sigma b_\sigma)^P$ and those of $(kG_{\sigma'} b_{\sigma'})^P$. This bijection is $N_{G_\sigma}(P)$-equivariant, so $N_{G_\sigma}(P_\gamma)=N_{G_{\sigma'}}(P_{\gamma'})$, and it follows that $(\sigma',P,\gamma',\pi,g)\in\mathbb{S}$. Now $(\sigma,P,\gamma,\pi,g)\mapsto (\sigma',P,\gamma',\pi,g)$ is an involution of the set $\mathbb{S}$, with the property that $|\sigma'|=|\sigma|\pm1$. Hence $|W(G,L,u)^\varphi|=0$ for any $\varphi\in\Aut(L,u)$ if $L\neq \un$. This completes the proof of Assertion~1.\medskip\par\noindent
2. It follows from Assertion~1 that FAwc$(G,b)$ and Awc$(G,b)$ are both equivalent to $n_{G,b}$ being equal to 1 if $d(b)=0$, and to 0 otherwise. Both parts of Assertion 2 follow.
\end{proof}

\begin{corollary}\label{cor stable vanishing}
Let $(G,b)$ be a group-block pair over $k$. Then the following stable version of Conjecture FAwc$(G,b)$ holds:
\begin{equation*}
  \sum_{\sigma\in [G\backslash \calS_p(G)]}  (-1)^{|\sigma|}  \, 
  \sur{\ldbrack G_\sigma, b_\sigma\rdbrack}_\F = 0\;\hbox{in}\;K_0(\sur{\FBL_{\F,k}})\,.
\end{equation*}
\end{corollary}

\begin{proof} Indeed, the simple functor $S_{\un,1,\F}$ becomes zero in the category $\sur{\FBL_{\F,k}}$ and the result follows from Theorem~\ref{thm Alperin reformulation}.
\end{proof}

\noindent\textbf{Acknowledgment} \ The first two authors are very grateful for the hospitality they experienced during their visits at the Mathematics Department at Bilkent University. The third author is supported by the Scientific and Technological Research Council of T{\"u}rkiye (T{\"U}B{\.I}TAK) under the 3501 Career Development Program with Project No. 123F456.


\bibliographystyle{abbrv}
\bibliography{for-bibtex}

\centerline{\rule{5ex}{.1ex}}
\begin{flushleft}
Robert Boltje, Department of Mathematics, University of California, Santa Cruz, 95064, California, USA.\\
{\tt boltje@ucsc.edu} \vspace{1ex}\\
Serge Bouc, CNRS-LAMFA, Universit\'e de Picardie, 33 rue St Leu, 80039, Amiens, France.\\
{\tt serge.bouc@u-picardie.fr}\vspace{1ex}\\
Deniz Y\i lmaz, Department of Mathematics, Bilkent University, 06800 Ankara, Turkey.\\
{\tt d.yilmaz@bilkent.edu.tr}
\end{flushleft}

%
%
%
%
%
%
%
%
%
%
\end{document}